\documentclass[12pt]{article}

\usepackage{amsmath, amsthm, amsfonts}
\usepackage[a4paper,top=2cm, bottom=2cm, left = 2cm, right=2cm]{geometry}
\usepackage{lmodern}        % Latin Modern family of fonts
\usepackage[T1]{fontenc}
\newtheorem{thm}{Theorem}[section]

\newtheorem{lem}[thm]{Lemma}

\theoremstyle{definition}

\theoremstyle{rem}
\newtheorem{rem}[thm]{Remark}

\title{\textbf{One-dimensional diffusion processes \\with moving membrane: partial \\reflection in combination with \\jump-like exit of process from \\membrane}}
\author{Bohdan~Kopytko\footnote{Czestochowa University of Technology, Poland. E-mail: bohdan.kopytko@im.pcz.pl}$\quad\quad$Roman~Shevchuk\footnote{Lviv Polytechnic National University,
    Ukraine. E-mail: r.v.shevchuk@gmail.com}
}
\date{}

\begin{document}
\maketitle
\abstract{We use analytical methods to construct the two-parameter Feller semigroup associated with a Markov process on a line with a moving membrane such that at the points on both sides of the membrane it coincides with the ordinary diffusion processes given there and its behavior after reaching the membrane is described by a kind of nonlocal Feller-Wentzell conjugation condition.}\\\\

In this paper, we construct the two-parameter Feller semigroup associated with a certain one-dimensional inhomogeneous Markov process and study its properties. We are interested in the process on the real line which can be described as follows. At the interior points of the half-lines separated by a point, the position of which depends on the time variable, this process coincides with the ordinary diffusion processes given there and its behavior on the common boundary of these half-lines is determined by a kind of the general Feller-Wentzell conjugation condition (see \cite{feller,langer,wentzell1}). This condition is nonlocal and contains only the terms that correspond to the property of partial reflection of the process and the possibility of its exit from the boundary of the domain by jumps. Such problems are often called problems of pasting together two diffusion processes on a line or problems of constructing the mathematical model of diffusion in medium with a membrane (see \cite{kopytko1,portenko2}). According to our understanding, the concept of a moving membrane means that its position on the real line changes and is determined by a given function which, as well as the process itself, depends on the time variable. The study of the problem is done using analytical methods. With such an approach, the problem of existence of the desired semigroup leads to the corresponding nonlocal conjugation problem for a second order linear parabolic equation of Kolmogorov's type to which the above problem is reduced. The peculiarity of this problem is that the half-bounded domains on the plane, where the equations are considered, are curvilinear. Furthermore, the function of the time variable which determines the common boundary of these domains, where two conjugation conditions are given (one of them represents the Feller property of the process and the other one corresponds to the nonlocal Feller-Wentzell conjugation condition for one-dimensional diffusion processes), satisfies only the H\"{o}lder condition with exponent greater than $\frac{1}{2}$. The nonclassical parabolic conjugation problem, formulated in such a way, is studied presumably for the first time. Its classical solvability is obtained by the boundary integral equations method with the use of the fundamental solutions of the Kolmogorov backward equations and the associated parabolic potentials (see \cite{friedman1,iljin,kopytko1,ladyzhenskaya,pogorzelski1} and \cite{baderko2}, where the review of the publications devoted to the application of this method to the investigation of initial-boundary value problems for parabolic equations is presented). Note that the main part of the proof of the existence theorem for this problem consists in justifying the possibility of applying the ordinary method of successive approximations (cf. \cite{kamynin1,kamynin2}) to the system of singular Volterra integral equations of the second kind which appears here, and the results presented in this part of our investigations are generalizations of the results obtained earlier in \cite{kopytko2,kopytko3} for the case of stationary membranes (see also \cite{shevchuk} where another kind of the general Feller-Wentzell boundary condition is considered). In the paper, in addition, some other properties of the constructed process are established for the first time. In particular, using the integral representation of the constructed semigroup, we succeed in finding its weak infinitesimal generator (cf. \cite{bogdan,xie}). It is shown that the class of Markov processes, extracted in the paper, covers as a partial case the generalized diffusion in understanding of the definition formulated in \cite{portenko1,portenko2}. Let us also pay attention to possible applications of generalized diffusion processes, as well as the methodology of their construction developed by us. E.g., in the introductory part of \cite{mastrangelo} it is noted that the diffusion processes admitting the generalized drift vector arise, in particular, in modeling physical processes which occur in the core of a nuclear reactor and in \cite{petruk} the boundary integral equations method is used in the study of a problem in high-energy astrophysics for constructing the classical solution of a particular case of non-stationary kinetic equation that describes the acceleration of charged particles in a vicinity of strong shock waves.

Finally, we note that there are many publications devoted to construction of diffusion processes (including multidimensional ones) with Wentzell boundary conditions using other approaches. Among them we note \cite{feller,langer,skubachevskii,taira,wentzell1,wentzell2} where there are presented results of application of the analytical approach to description of the mentioned class of homogeneous Markov processes based on methods of the semigroup theory and functional analysis in relation to the elliptic boundary value problems, and \cite{anulova,engelbert,etore,ikeda,kulik,kulinich,pilipenko,skorokhod,walsh,zaitseva} which partially reflect the development of methods of stochastic analysis for the construction of such type of processes (see also the references given there). We also mention the paper \cite{lejay} where the problem of pasting together one-dimensional diffusions is formulated in a slightly different form (compared to the one we use) and with somewhat different (in comparison with those considered in this paper) boundary conditions.

Our paper is organized as follows.

In Section 1, we formulate the problem of pasting together two diffusion processes in terms of the parabolic problem of conjugation and provide a review of auxiliary results on the fundamental solution of the backward Kolmogorov equation and the associated potentials which will be used in the subsequent sections.

Section 2 is devoted to the proof of the existence and uniqueness theorems for the conjugation problem formulated in Section 1.

In Section 3, using the solution of this parabolic conjugation problem, we define the two-parameter Feller semigroup which describes the desired Markov process. Some additional properties of the constructed process are also investigated here.

\section{Setting of the nonlocal parabolic conjugation problem, main assumptions and auxiliary results}

Consider the strip
\[
S_t=\{(s,x):\quad 0\leq s<t\leq T,\, -\infty<x<\infty\}
\]
in the plane \((s,x)\in\mathbb{R}^2\). Assume that \(S_T\) contains a continuous curve \(x=h(s),\, 0\leq s\leq T,\) which separates \(S_t\) into two domains

\[
S_t^{(1)}=\{(s,x):\quad 0\leq s<t\leq T,\, -\infty<x<h(s)\}
\]
and
\[
S_t^{(2)}=\{(s,x):\quad 0\leq s<t\leq T,\, h(s)<x<\infty\}.
\]

Put \(D_{1s}=(-\infty,h(s)),\,D_{2s}=(h(s),\infty),\,D_s=D_{1s}\cup D_{2s},\,D_{is}^{\delta}=\{y:\, y\in D_{is},\,|y-h(s)|<\delta\},\,i=1,2,\,D_s^{\delta}=D_{1s}^{\delta}\cup D_{2s}^{\delta}\) where \(\delta\) is any positive number. Denote by \(\overline{G}\) the closure of a set \(G\).

In the strip \(S_T\), we consider two uniformly parabolic operators with the bounded continuous coefficients
\begin{equation}\label{eq:ParabolicOp}
\frac{\partial}{\partial s}+L_s^{(i)}\equiv\frac{\partial}{\partial s}+\frac{1}{2}b_i(s,x)\frac{\partial^2}{\partial x^2}+a_i(s,x)\frac{\partial}{\partial x},\quad i=1,2.
\end{equation}

The problem is to find a classical solution \(u(s,x,t)\) of the equation
\begin{equation}\label{eq:ParabolicEq}
\frac{\partial u}{\partial s}+L_s^{(i)}u=0,\quad (s,x)\in S_t^{(i)},\text{ }i=1,2,
\end{equation}
which satisfies the "initial" condition
\begin{equation}\label{eq:InitialCond}
\lim_{s\uparrow t}u(s,x,t)=\varphi(x),\quad x\in\mathbb{R},
\end{equation}
and two conjugation conditions
\begin{equation}\label{eq:FellerProperty}
B_1u\equiv u(s,h(s)-0,t)-u(s,h(s)+0,t)=0,\quad 0\leq s\leq t\leq T,
\end{equation}
\begin{equation}\label{eq:FellerWentzell}\begin{split}
B_2u&\equiv q_1(s)\frac{\partial u}{\partial x}(s,h(s)-0,t)-q_2(s)\frac{\partial u}{\partial x}(s,h(s)+0,t)+\\
&+\int\limits_{D_s}\left[u(s,h(s),t)-u(s,y,t)\right]\mu(s,dy)=0,\quad 0\leq s<t\leq T,
\end{split}\end{equation}
where \(\varphi\) is a function which is bounded and continuous on \(\mathbb{R}\), \(q_i,\,i=1,2,\) are nonnegative continuous on \([0,T]\) functions such that
\begin{equation}\label{ineq:ReflectionCoef}
q_1(s)+q_2(s)>0,\quad s\in[0,T],
\end{equation}
and \(\mu(s,\cdot)\) is a nonnegative Borel measure on \(D_s\) such that for any \(\delta>0,\)
\begin{equation}\label{ineq:InfiniteMeasure}
\int\limits_{D_s^{\delta}}|y-h(s)|\mu(s,dy)+\mu(s,D_s\setminus D_s^{\delta})<\infty,\quad s\in[0,T].
\end{equation}
Here, \(u(s,h(s)-0,t)\,\left(\frac{\partial u}{\partial x}(s,h(s)-0,t)\right)\) and \(u(s,h(s)+0,t)\,\left(\frac{\partial u}{\partial x}(s,h(s)+0,t)\right)\) denote the limits of the function \(u(s,x,t)\,\left(\frac{\partial u}{\partial x}(s,x,t)\right)\) at \((s,h(s))\) as the point \((s,x)\) tends to \((s,h(s))\) from the side of the domains \(S_t^{(1)}\) and \(S_t^{(2)}\) respectively.

In view of the above setting of the problem \eqref{eq:ParabolicEq}-\eqref{eq:FellerWentzell}, it is convenient to write the function \(u\) in the form
\[u(s,x,t)=\begin{cases}u_1(s,x,t)\quad\text{if}\quad(s,x)\in \overline{S}_t^{(1)},\\
u_2(s,x,t)\quad\text{if}\quad(s,x)\in\overline{S}_t^{(2)}.\end{cases}\]
Thus, to solve the problem \eqref{eq:ParabolicEq}-\eqref{eq:FellerWentzell}, we can find a pair of functions \(u_i(s,x,t)\,(i=1,2),\) continuous and bounded in the closed domain \(\overline{S}_t^{(i)},\) which satisfy the equation \eqref{eq:ParabolicEq} in the open domain \(S_t^{(i)}\) (in the classical meaning), the "initial" condition \eqref{eq:InitialCond} for \(x\in \overline{D}_{it},\) and, for \(x=h(s),\) they are interrelated by two conjugation conditions \eqref{eq:FellerProperty} and \eqref{eq:FellerWentzell}, where \(B_1u\) and \(B_2u\) can be replaced by the following expressions:
\begin{align}
B_1u&\equiv u_1(s,h(s),t)-u_2(s,h(s),t),\nonumber\\
B_2u&\equiv q_1(s)\frac{\partial u_1}{\partial x}(s,h(s),t)-q_2(s)\frac{\partial u_2}{\partial x}(s,h(s),t)+\nonumber\\
&+\sum\limits_{i=1}^2\int\limits_{D_{is}}\left[u_i(s,h(s),t)-u_i(s,y,t)\right]\mu(s,dy).\nonumber
\end{align}

Again, we note that the problem \eqref{eq:ParabolicEq}-\eqref{eq:FellerWentzell} appears, in particular, in the theory of diffusion processes when constructing the one-dimensional model of diffusion with membrane or, what is the same, when solving (using analytical methods) the so-called problem of pasting together two diffusion processes on a line. In the case under consideration, the membrane is assumed to be moving and it is placed at the point \(x=h(s)\), i.e., at the point of pasting together two given diffusion processes. In the paper we prove that there exists a unique solution of \eqref{eq:ParabolicEq}-\eqref{eq:FellerWentzell} and that the family of operators \(T_{st}\varphi(x)\equiv u(s,x,t)\) is the two-parameter Feller semigroup associated with an inhomogeneous Markov process on the line. Then the fulfillment of the equation \eqref{eq:ParabolicEq} for the function \(T_{st}\varphi(x)\) means that this process coincides in \(D_{is}\) with the diffusion process given there by generating operator $L_s^{(i)},\,i=1,2,$ and the "initial" condition \eqref{eq:InitialCond} is in agreement with the equality \(T_{ss}=I,\) where \(I\) is the identity operator. Next, the conjugation condition \eqref{eq:FellerProperty} reflects the Feller property of the process and the equality \eqref{eq:FellerWentzell} represents one of kinds of the general Feller-Wentzell conjugation condition (see \cite{kopytko2,langer,wentzell1}) which includes only the terms corresponding to partial reflection of the process at the point where the membrane is placed and to the possibility of exit of the process from this point by jumps. Recall that the most general Feller-Wentzell boundary condition contains two more terms: the unknown function and its derivative with respect to the time variable that correspond to such properties of the process on the common boundary of the domains as its termination and delay.

We need the following conditions:
\begin{enumerate}
\item[I.] The equation \eqref{eq:ParabolicEq} is the equation of parabolic type in \(\overline{S}_T\), i.e., there exist positive constants \(b\) and \(B\) such that \(0<b\leq b_i(s,x)\leq B<\infty,\,i=1,2,\,(s,x)\in \overline{S}_T.\)
\item[II.] The coefficients \(a_i(s,x)\) and \(b_i(s,x),\,i=1,2,\) are continuous in \(\overline{S}_T\) and belong to the H\"{o}lder class \(H^{\frac{\alpha}{2},\alpha}(\overline{S}_T),\,0<\alpha<1\) (for the definition of H\"{o}lder classes, see \cite[Ch.I, \S1]{ladyzhenskaya}).
\item[III.] The initial function \(\varphi\) in \eqref{eq:InitialCond} belongs to the space of bounded and continuous functions, that will be denoted by \(C_b(\mathbb{R})\). The norm in this space is defined by \(\|\varphi\|=\sup\limits_{x\in\mathbb{R}}|\varphi(x)|.\)
\item[IV.] The functions \(q_i(s),\,i=1,2,\,s\in[0,T]\) in \eqref{eq:FellerWentzell} are nonnegative, continuous and satisfy the inequality \eqref{ineq:ReflectionCoef}; the measure \(\mu(s,\cdot)\) is nonnegative, satisfies the inequality \eqref{ineq:InfiniteMeasure} and, for any function \(f\in C_b(\mathbb{R})\) and any number \(\delta>0,\) the integrals
    \[F_f^{(i)}(s)=\int\limits_{D_{is}^{\delta}}(y-h(s))f(y)\mu(s,dy),\quad G_f^{(i)}(s)=\int\limits_{D_{is}\setminus D_{is}^{\delta}}f(y)\mu(s,dy),\quad i=1,2,\] are continuous on \([0,T]\) as functions of \(s\).
\item[V.] The curve \(h(s)\) is continuous and belongs to the H\"{o}lder class \(H^{\frac{1+\alpha}{2}}([0,T])\).
\end{enumerate}

The conditions I, II ensure the existence of the fundamental solution for each equation in \eqref{eq:ParabolicEq} (see \cite[Ch.I, \S6]{friedman1}, \cite[Ch.IV, \S11]{ladyzhenskaya}, \cite[Ch.II, \S2]{portenko2}), i.e., the existence of the function \(G_i(s,x,t,y)\,(i=1,2)\) defined for all \(0\leq s<t\leq T\) and \(x,y\in\mathbb{R}\), which has the following properties:
\begin{itemize}
\item For fixed \(t\in[0,T)\) and \(y\in\mathbb{R}\) the function \(G_i\) satisfies the equation \eqref{eq:ParabolicEq} as a function of \((s,x)\in[0,t)\times\mathbb{R}\).
\item For any bounded continuous function \(\varphi(x)\) in \(\mathbb{R}\),
\[\lim\limits_{s\uparrow t}\int\limits_{\mathbb{R}}G_i(s,x,t,y)\varphi(y)dy=\varphi(x)\]
if \(t\in(0,T],\,x\in\mathbb{R}.\)
\end{itemize}
Furthermore, the function \(G_i(s,x,t,y)\) allows the representation (see \cite[Ch.I, \S2]{friedman1}, \cite[Ch.IV, \S11]{ladyzhenskaya}, \cite[Ch.II, \S2]{portenko2})
\begin{equation}\label{eq:FS}
G_i(s,x,t,y)=Z_{i0}(s,x,t,y)+Z_{i1}(s,x,t,y),\quad i=1,2,
\end{equation}
where
\begin{align}
&Z_{i0}(s,x,t,y)=[2\pi b_i(t,y)(t-s)]^{-\frac{1}{2}}\exp\left\{-\frac{(y-x)^2}{2b_i(t,y)(t-s)}\right\},\label{eq:FS1}\\
&Z_{i1}(s,x,t,y)=\int\limits_s^t d\tau\int\limits_{\mathbb{R}}Z_{i0}(s,x,\tau,z)Q_i(\tau,z,t,y)dz,\label{eq:FS2}
\end{align}
and the function \(Q_i(s,x,t,y)\) is the solution of some singular Volterra integral equation of the second kind.

Note that (see \cite[Ch.IV, \S13]{ladyzhenskaya})
\begin{align}
&|D_s^rD_x^pG_i(s,x,t,y)|\leq C(t-s)^{-\frac{1+2r+p}{2}}\exp\left\{-c\frac{(y-x)^2}{t-s}\right\},\label{ineq:FSEstimate1}\\
&|D_s^rD_x^pZ_{i1}(s,x,t,y)|\leq C(t-s)^{-\frac{1+2r+p-\alpha}{2}}\exp\left\{-c\frac{(y-x)^2}{t-s}\right\},\label{ineq:FSEstimate2}
\end{align}
where \(i=1,2,\,0\leq s<t\leq T,\,x,y,\in\mathbb{R},\,C\) and \(c\) are positive constants (in what follows, various positive constants depending on the data of the problem \eqref{eq:ParabolicEq}-\eqref{eq:FellerWentzell} will be denoted by \(C\) or \(c\) without specifying their values); \(r\) and \(p\) are the nonnegative integers for which \(2r+p\leq 2,\, D_s^r\) is the partial derivative with respect to \(s\) of order \(r,\,D_x^p\) is the partial derivative with respect to \(x\) of order \(p\).
\begin{rem}[The estimate for \(Z_{i0}\)]
The inequality \eqref{ineq:FSEstimate1} holds for any nonnegative integers \(r\) and \(p\), when \(G_i\) is replaced by \(Z_{i0}\).
\end{rem}

Note also that (see \cite[Ch.II, \S2]{portenko1}, \cite[Ch.II, \S2]{portenko2})
\begin{align}
&\label{eq:FSadditional1}\int\limits_{\mathbb{R}}G_i(s,x,t,y)dy=1,\quad i=1,2,\\
&\label{eq:FSadditional2}\int\limits_{\mathbb{R}}G_i(s,x,t,y)(y-x)dy=\int\limits_s^td\tau\int\limits_{\mathbb{R}}G_i(s,x,\tau,z)a_i(\tau,z)dz,\quad i=1,2,\\
\label{eq:FSadditional3}\begin{split}&\int\limits_{\mathbb{R}}G_i(s,x,t,y)(y-x)^2dy=\int\limits_s^td\tau\int\limits_{\mathbb{R}}G_i(s,x,\tau,z)b_i(\tau,z)dz+\\
&+2\int\limits_s^t d\tau\int\limits_{\mathbb{R}}G_i(s,x,\tau,z)a_i(\tau,z)(z-x)dz,\quad i=1,2.\end{split}
\end{align}

Given the fundamental solution \(G_i(s,x,t,y),\,i=1,2,\) and the functions \(\varphi(x),\,x\in\mathbb{R},\,h(s),\,s\in[0,T],\) consider the integrals:
\begin{align}
&u_{i0}(s,x,t)=\int\limits_{\mathbb{R}}G_i(s,x,t,y)\varphi(y)dy,\quad i=1,2,\label{eq:Poisson}\\
&u_{i1}(s,x,t)=\int\limits_s^tG_i(s,x,\tau,h(\tau))V_i(\tau,t)d\tau,\quad i=1,2\label{eq:Simple-layer}
\end{align}
(\(V_i(s,t),\,i=1,2,\) is a function defined for \(0\leq s<t\leq T\)). In the theory of parabolic equations, the functions \(u_{i0}(s,x,t)\) and \(u_{i1}(s,x,t)\) are called the Poisson potential and the parabolic simple-layer potential, respectively.

Let us note some properties of the functions \(u_{i0}\) and \(u_{i1},\,i=1,2\). Recall that \(\varphi\in C_b(\mathbb{R})\). From the definition of the fundamental solution \(G_i,\,i=1,2,\) and the estimates \eqref{ineq:FSEstimate1} and \eqref{ineq:FSEstimate2}, it follows that the potential \(u_{i0}(s,x,t)\) exists and, as a function of \((s,x),\) for fixed \(t\in(0,T]\) satisfies the equation \eqref{eq:ParabolicEq} in \((s,x)\in[0,t)\times\mathbb{R}\) with the "initial" condition
\begin{equation}\label{eq:PoissonInitialCond}
\lim\limits_{s\uparrow t}u_{i0}(s,x,t)=\varphi(x),\quad x\in\mathbb{R},\, i=1,2.
\end{equation}
Furthermore, the function \(u_{i0}(s,x,t),\,i=1,2,\) satisfies the estimate
\begin{equation}\label{ineq:PoissonEstimate}
|D_s^rD_x^p u_{i0}(s,x,t)|\leq C\|\varphi\|(t-s)^{-\frac{2r+p}{2}},
\end{equation}
where \(0\leq s<t\leq T,\,x\in\mathbb{R},\,r\) and \(p\) are the nonnegative integers such that \(2r+p\leq 2\).

Consider the integral \eqref{eq:Simple-layer}. If we suppose that the density \(V_i(\tau,t)\) is continuous in \(0\leq\tau<t\leq T\) and allows the inequality
\begin{equation}\label{ineq:SingularityV}
|V_i(\tau,t)|\leq C(t-\tau)^{-\mu},\quad 0\leq\mu<1,\,i=1,2,
\end{equation}
in this domain, then the function \(u_{i1}(s,x,t),\,i=1,2,\) is continuous in \(0\leq s<t\leq T,\,x\in\mathbb{R}\) and satisfies the equation \eqref{eq:ParabolicEq} in the domain \((s,x)\in[0,t)\times(\mathbb{R}\setminus h(s))\) and the initial condition
\begin{equation}\label{eq:ui1Initial}
\lim\limits_{s\uparrow t}u_{i1}(s,x,t)=0,\quad i=1,2,
\end{equation}
at every point \(x\in\mathbb{R}\setminus{h(t)}\). Furthermore, if \(0\leq\mu<\frac{1}{2}\), then this function is bounded continuous in \(0\leq s\leq t\leq T,\,x\in\mathbb{R}\) and \eqref{eq:ui1Initial} holds for all \(x\in\mathbb{R}\). Note also that the problem of validity of the same property for \(u_{i1}(s,x,t),\,i=1,2,\) in the case \(\mu=\frac{1}{2}\) requires an additional verification for every individual situation.

An important property of the function \(u_{i1}\) is described by the so-called theorem on the jump of conormal derivative of a parabolic simple-layer potential (see, e.g., \cite{baderko1}, \cite[Ch.V, \S2]{friedman1}, \cite[Ch.IV, \S15]{ladyzhenskaya} as well as \cite[Ch.XXII, \S8]{pogorzelski2} for more general results). For the potential \(u_{i1}(s,x,t)\) under consideration, this theorem asserts that if the curve \(h(s)\) belongs to \(H^{\frac{1+\alpha}{2}}([0,T])\) for some \(\alpha\in(0,1)\), the function \(V_i(\tau,t)\) is continuous in \(0\leq\tau<t\leq T\) and satisfies \eqref{ineq:SingularityV} in this domain, then for every point \(x=h(s),\,s\in[0,t)\) the function \(u_{i1}(s,x,t)\) satisfies the relation
\begin{equation}\label{eq:JumpFormula}
\lim\limits_{x\rightarrow h(s)\pm0}\frac{\partial u_{i1}}{\partial x}(s,x,t)=\mp\frac{V_i(s,t)}{b_i(s,h(s))}+\int\limits_s^t\frac{\partial G_i}{\partial x}(s,h(s),\tau,h(\tau))V_i(\tau,t)d\tau,\quad i=1,2.
\end{equation}
The integral on the right-hand side of \eqref{eq:JumpFormula} is called the direct value of a simple-layer potential. Its existence follows from the inequality
\begin{equation}\label{ineq:DxGiEstimate}
\left|\frac{\partial G_i}{\partial x}(s,h(s),\tau,h(\tau))\right|\leq C(\tau-s)^{-1+\frac{\alpha}{2}},\quad i=1,2.
\end{equation}

\section{Solving the nonlocal parabolic conjugation problem using the boundary integral equations method}
In this section we prove the existence and uniqueness theorems for the conjugation problem \eqref{eq:ParabolicEq}-\eqref{eq:FellerWentzell} (Theorems \ref{th:Existence} and \ref{th:Uniqueness}). We begin with
\begin{thm}[Existence]\label{th:Existence} Let the conditions I-IV hold. Then the nonlocal parabolic conjugation problem \eqref{eq:ParabolicEq}-\eqref{eq:FellerWentzell} has a classical solution which is continuous in \(\overline{S}_t\).
\end{thm}
\begin{proof}
We look for a solution \(u(s,x,t)=u_i(s,x,t),\,(s,x)\in \overline{S}_t^{(i)},\,i=1,2,\) of the problem \eqref{eq:ParabolicEq}-\eqref{eq:FellerWentzell} of the form
\begin{equation}\label{eq:Solution}
u_i(s,x,t)=u_{i0}(s,x,t)+u_{i1}(s,x,t),\quad (s,x)\in \overline{S}_t^{(i)},\,i=1,2,
\end{equation}
where the functions \(u_{i0}\) and \(u_{i1}\) are defined by the formulas \eqref{eq:Poisson} and \eqref{eq:Simple-layer}, respectively, in which \(\varphi\) is the "initial" function in \eqref{eq:InitialCond} and \(V_i\) are unknown functions to be found.

Suppose a priori that the unknown densities \(V_i(s,t),\,i=1,2,\) are continuous in the domain \(0\leq s<t\leq T\) and satisfy \eqref{ineq:SingularityV} with \(0\leq\mu\leq\frac{1}{2}\) in this domain. Using the conjugation conditions \eqref{eq:FellerProperty} and \eqref{eq:FellerWentzell}, in view of \eqref{eq:JumpFormula}, we obtain the following system of two Volterra integral equations for \(V_i\):
\begin{align}
&\sum\limits_{i=1}^2 (-1)^{i-1}\int\limits_s^t G_i(s,h(s),\tau, h(\tau))V_i(\tau,t)d\tau=\Phi_0(s,t),\label{eq:VolterraI}\\
&\sum\limits_{i=1}^2 \frac{q_i(s)}{b_i(s,h(s))}V_i(s,t)=\Psi(s,t)+\sum\limits_{j=1}^2\int\limits_s^t K_j(s,\tau)V_j(\tau,t)d\tau,\label{eq:VolterraII}
\end{align}
where
\[
\begin{split}\nonumber
\Phi_0(s,t)&=u_{20}(s,h(s),t)-u_{10}(s,h(s),t),\\
\Psi(s,t)&=q_2(s)\frac{\partial u_{20}}{\partial x}(s,h(s),t)-q_1(s)\frac{\partial u_{10}}{\partial x}(s,h(s),t)+\\
&+\sum\limits_{i=1}^2\int\limits_{D_{is}}[u_{i0}(s,y,t)-u_{i0}(s,h(s),t)]\mu(s,dy),\\
K_j(s,\tau)&=(-1)^j q_j(s)\frac{\partial G_j}{\partial x}(s,h(s),\tau,h(\tau))+\\
&+\int\limits_{D_{js}}[G_j(s,y,\tau,h(\tau))-G_j(s,h(s),\tau,h(\tau))]\mu(s,dy),\quad j=1,2.
\end{split}
\]

We see that the above system of equations, in addition to the Volterra integral equation of the second kind \eqref{eq:VolterraII}, contains also the Volterra integral equation of the first kind \eqref{eq:VolterraI}. Using the Holmgren transform (see \cite{kamynin1,kamynin2,kopytko2,kopytko3,shevchuk}), we reduce this equation to the equivalent Volterra integral equation of the second kind. For this purpose we introduce the integro-differential operator \(\mathcal{E}\) as follows:
\begin{equation}\label{eq:Holmgren}
\mathcal{E}f(s,t)=\sqrt{\frac{2}{\pi}}\frac{\partial}{\partial s}\int_s^t(\rho-s)^{-\frac{1}{2}}f(\rho,t)d\rho,\quad 0\leq s<t\leq T.
\end{equation}
Consider first the application of the operator \(\mathcal{E}\) to the right-hand side of \eqref{eq:VolterraI}, i.e., to the function \(\Phi_0(s,t)\). From the properties of Poisson potentials noted in Section 1, it follows that the function \(\Phi_0(s,t)\) is continuous in \(\overline{S}_t\) and satisfies the condition
\begin{equation}\label{eq:LimPhi0}
\lim\limits_{s\uparrow t}\Phi_0(s,t)=0.
\end{equation}
Furthermore, from the finite-increments formula and from the estimates \eqref{ineq:PoissonEstimate}, it easily follows that
\begin{equation}\label{ineq:Phi0Estimate}
|\Phi_0(s,t)-\Phi_0(\widetilde{s},t)|\leq C\|\varphi\|(t-s)^{-\frac{1+\alpha}{2}}(s-\widetilde{s})^{\frac{1+\alpha}{2}},\quad \widetilde{s}<s,\,\alpha\in (0,1).
\end{equation}
Let \(\Phi(s,t)=\mathcal{E}\Phi_0(s,t)\). In view of \eqref{eq:LimPhi0}, \eqref{ineq:Phi0Estimate}, we obtain the formula
\begin{equation}\label{eq:PhiFormula}
\Phi(s,t)=\frac{1}{\sqrt{2\pi}}\int\limits_s^t(\rho-s)^{-\frac{3}{2}}[\Phi_0(\rho,t)-\Phi_0(s,t)]d\rho-\sqrt{\frac{2}{\pi}}(t-s)^{-\frac{1}{2}}\Phi_0(s,t)
\end{equation}
and the following estimate:
\begin{equation}\label{ineq:PhiEstimate}
|\Phi(s,t)|\leq C\|\varphi\|(t-s)^{-\frac{1}{2}},\quad 0\leq s<t\leq T.
\end{equation}

Denote by \(I(s,t)\) the expression on the left-hand side of \eqref{eq:VolterraI}. Applying the operator \(\mathcal{E}\) to \(I(s,t)\), we get
\[
\mathcal{E}I(s,t)=\sqrt{\frac{2}{\pi}}\frac{\partial}{\partial s}\int\limits_s^t(\rho-s)^{-\frac{1}{2}}\bigg[\sum\limits_{i=1}^2(-1)^{i-1}\int\limits_s^t G_i(\rho,h(\rho),\tau,h(\tau))V_i(\tau,t)d\tau\bigg]d\rho.
\]

In order to take the derivative with respect to \(s\) in the last expression, we use the representation
\begin{align}
G_i(\rho,h(\rho),\tau,h(\tau))&=[G_i(\rho,h(\rho),\tau,h(\tau))-G_i(\rho,h(\tau),\tau,h(\tau))]+\nonumber\\
&+G_i(\rho,h(\tau),\tau,h(\tau)),\quad i=1,2.\nonumber
\end{align}
Interchanging the order of integration and using the finite-increments formula for the difference \(G_i(\rho,h(\rho),\tau,h(\tau))-G_i(\rho,h(\tau),\tau,h(\tau))\), the condition V, the representation \eqref{eq:FS} for the function \(G_i(\rho,h(\tau),\tau,h(\tau))\), the relation
\[
\int\limits_s^{\tau}(\tau-\rho)^{-\frac{1}{2}}(\rho-s)^{-\frac{1}{2}}d\rho=\pi
\]
and the estimates \eqref{ineq:FSEstimate1} and \eqref{ineq:FSEstimate2}, we find that
\begin{equation}\label{eq:E(s,t)I}
\begin{split}
\hat{I}(s,t)&=\mathcal{E}I(s,t)=\sum\limits_{i=1}^2(-1)^i\frac{V_i(s,t)}{\sqrt{b_i(s,h(s))}}-\\
&-\sum\limits_{j=1}^2\int\limits_s^t R_j(s,\tau)V_j(\tau,t)d\tau,\quad 0\leq s<t\leq T,
\end{split}
\end{equation}
where
\[
\begin{split}
R_j(s,\tau)&=(-1)^j\sqrt{\frac{2}{\pi}}\frac{\partial}{\partial s}\int\limits_s^{\tau}(\rho-s)^{-\frac{1}{2}}[Z_{j1}(\rho,h(\tau),\tau,h(\tau))+\\
&+(G_j(\rho,h(\rho),\tau,h(\tau))-G_j(\rho,h(\tau),\tau,h(\tau)))]d\rho,\quad j=1,2.
\end{split}
\]
Furthermore, for the kernels \(R_j(s,\tau),\,j=1,2,\) the inequality
\begin{equation}\label{ineq:RjEstimate}
|R_j(s,\tau)|\leq C(\tau-s)^{-1+\frac{\alpha}{2}}
\end{equation}
holds when \(0\leq s<\tau\leq t\leq T.\)

Let us prove the inequality \eqref{ineq:RjEstimate}. We will use the notations \[\Delta_t^{\tau}f(t,x)\equiv f(t,x)-f(\tau,x),\quad\Delta_x^{y}f(t,x)\equiv f(t,x)-f(t,y)\] and the following representation for \(R_j(s,\tau),\,j=1,2:\)
\begin{equation}\label{eq:RjRepresentation}
\begin{split}
R_j(s,\tau)&=\frac{(-1)^j}{\sqrt{2\pi}}\int\limits_s^{\tau}(\rho-s)^{-\frac{3}{2}}\big[\Delta_{\rho}^sZ_{j1}(\rho,h(\tau),\tau,h(\tau))+\\
&+\Delta_{\rho}^s\Delta_{h(\rho)}^{h(\tau)}G_j(\rho,h(\rho),\tau,h(\tau))+\Delta_{h(\rho)}^{h(s)}G_j(s,h(\rho),\tau,h(\tau))\big]d\rho+\\
&+(-1)^{j-1}\sqrt{\frac{2}{\pi}}(\tau-s)^{-\frac{1}{2}}\big[Z_{j1}(s,h(\tau),\tau,h(\tau))+\\
&+\Delta_{h(s)}^{h(\tau)}G_j(s,h(s),\tau,h(\tau))\big]=R_j^{(1)}(s,\tau)+R_j^{(2)}(s,\tau).
\end{split}
\end{equation}

To estimate \(R_j^{(2)}(s,\tau)\), we use the finite-increments formula for the difference \[\Delta_{h(s)}^{h(\tau)}G_j(s,h(s),\tau,h(\tau)),\] the condition V, and the inequalities \eqref{ineq:FSEstimate1} and \eqref{ineq:FSEstimate2}. We obtain
\begin{equation}\label{ineq:Rj2Estimate}
|R_j^{(2)}(s,\tau)|\leq C (\tau-s)^{-\frac{1}{2}}\left[(\tau-s)^{-\frac{1}{2}+\frac{\alpha}{2}}+(\tau-s)^{-1}(\tau-s)^{\frac{1+\alpha}{2}}\right]\leq C(\tau-s)^{-1+\frac{\alpha}{2}}.
\end{equation}

Using the representation \eqref{eq:FS}, the mean value theorem for the functions \(G_i,\,Z_{i0},\,Z_{i1},\) \(i=1,2,\) the H\"{o}lder condition for the function \(h\) and the inequalities \eqref{ineq:FSEstimate1} and \eqref{ineq:FSEstimate2}, we can estimate the integrands and the integral itself, which represents the function \(R_j^{(1)}(s,\tau)\). To do so, we split this integral into two parts integrated over \((s,\frac{s+\tau}{2})\) and \((\frac{s+\tau}{2},\tau),\) respectively. We use the inequalities
\begin{equation}\nonumber
\begin{split}
&(\tau-\rho)\geq\frac{1}{2}(\tau-s),\quad |\Delta_{\rho}^s Z_{j1}(\rho,h(\tau),\tau,h(\tau))|\leq C(\rho-s)(\tau-\rho)^{-\frac{3}{2}+\frac{\alpha}{2}},\\
&|\Delta_{\rho}^s\Delta_{h(\rho)}^{h(\tau)} G_j(\rho,h(\rho),\tau,h(\tau))|\leq|\Delta_{\rho}^s\Delta_{h(\rho)}^{h(\tau)} Z_{j0}(\rho,h(\rho),\tau,h(\tau))|+\\
&+|\Delta_{\rho}^s\Delta_{h(\rho)}^{h(\tau)} Z_{j1}(\rho,h(\rho),\tau,h(\tau))|\leq C(\rho-s)[(\tau-\rho)^{-2}(\tau-\rho)^{\frac{1+\alpha}{2}}+(\tau-\rho)^{-\frac{3}{2}+\frac{\alpha}{2}}]\leq\\
&\leq C(\rho-s)(\tau-\rho)^{-\frac{3}{2}+\frac{\alpha}{2}},\quad|\Delta_{h(\rho)}^{h(s)}G_j(s,h(\rho),\tau,h(\tau))|\leq C(\tau-s)^{-1}(\rho-s)^{\frac{1+\alpha}{2}}
\end{split}
\end{equation}
in the first integral and the inequalities
\begin{equation}
\begin{split}\nonumber
&\rho-s\geq\frac{1}{2}(\tau-s),\quad |\Delta_{\rho}^s Z_{j1}(\rho,h(\tau),\tau,h(\tau))|\leq C[(\tau-\rho)^{-\frac{1}{2}+\frac{\alpha}{2}}+(\tau-s)^{-\frac{1}{2}+\frac{\alpha}{2}}],\\
&|\Delta_{\rho}^s\Delta_{h(\rho)}^{h(\tau)}G_j(\rho,h(\rho),\tau,h(\tau))|\leq |\Delta_{h(\rho)}^{h(\tau)}G_j(\rho,h(\rho),\tau,h(\tau))|+|\Delta_{h(\rho)}^{h(\tau)}G_j(s,h(\rho),\tau,h(\tau))|\leq\\
&\leq C[(\tau-\rho)^{-\frac{1}{2}+\frac{\alpha}{2}}+(\tau-s)^{-1}(\tau-\rho)^{\frac{1+\alpha}{2}}],\\
&|\Delta_{h(\rho)}^{h(s)}G_j(s,h(\rho),\tau,h(\tau))|\leq C(\tau-s)^{-1}(\rho-s)^{\frac{1+\alpha}{2}}
\end{split}
\end{equation}
in the second one.

Consequently,
\begin{equation}\label{ineq:Rj1Estimate}
\begin{split}
&|R_j^{(1)}(s,\tau)|\leq C\bigg(\int\limits_s^{\frac{s+\tau}{2}}[(\rho-s)^{-\frac{1}{2}}(\tau-s)^{-\frac{3}{2}+\frac{\alpha}{2}}+(\rho-s)^{-1+\frac{\alpha}{2}}(\tau-s)^{-1}]d\rho+\\
&+\int\limits_{\frac{s+\tau}{2}}^{\tau}(\tau-s)^{-\frac{3}{2}}[(\tau-\rho)^{-\frac{1}{2}+\frac{\alpha}{2}}+(\tau-s)^{-\frac{1}{2}+\frac{\alpha}{2}}+(\tau-s)^{-1}(\tau-\rho)^{\frac{1+\alpha}{2}}+\\
&+(\tau-s)^{-1}(\rho-s)^{\frac{1+\alpha}{2}}]d\rho\bigg)\leq C(\tau-s)^{-1+\frac{\alpha}{2}},\quad j=1,2.
\end{split}
\end{equation}

Combining \eqref{ineq:Rj2Estimate} with \eqref{ineq:Rj1Estimate}, we conclude that the inequality (\ref{ineq:RjEstimate}) holds.

From \eqref{eq:PhiFormula} and \eqref{eq:E(s,t)I}, it follows that the application of the operator \(\mathcal{E}\) to the both sides of \eqref{eq:VolterraI} leads to the Volterra integral equation of the second kind
\begin{equation}\label{eq:VolterraIReplacedByII}
\sum\limits_{i=1}^2(-1)^i\frac{V_i(s,t)}{\sqrt{b_i(s,h(s))}}=\Phi(s,t)+\sum\limits_{j=1}^2\int\limits_s^t R_j(s,\tau)V_j(\tau,t)d\tau.
\end{equation}

Thus, the system of equations \eqref{eq:VolterraI} and \eqref{eq:VolterraII} can be replaced by the equivalent system of Volterra integral equations of the second kind \eqref{eq:VolterraIReplacedByII} and \eqref{eq:VolterraII} which, after simple transformations, can be represented as
\begin{equation}\label{eq:SystemVolterraEqII}
V_i(s,t)=\Psi_i(s,t)+\sum\limits_{j=1}^2\int\limits_s^t N_{ij}(s,\tau)V_j(\tau,t)d\tau,\quad i=1,2,
\end{equation}
where
\[
\begin{split}
&\Psi_i(s,t)=d_i(s)\left[\Psi(s,t)+\frac{(-1)^i q_{3-i}(s)}{\sqrt{b_{3-i}(s,h(s))}}\Phi(s,t)\right],\quad i=1,2,\\
&N_{ij}(s,\tau)=d_i(s)\left[K_j(s,\tau)+\frac{(-1)^iq_{3-i}(s)}{\sqrt{b_{3-i}(s,h(s))}}R_j(s,\tau)\right],\quad i,j=1,2,\\
&d_i(s)=\frac{b_i(s,h(s))\sqrt{b_{3-i}(s,h(s))}}{q_1(s)\sqrt{b_2(s,h(s))}+q_2(s)\sqrt{b_1(s,h(s))}},\quad i=1,2.
\end{split}
\]

Let us investigate the properties of the functions \(\Psi_i(s,t),\,i=1,2,\) and the singularities of the kernels \(N_{ij}(s,\tau),\,i,j=1,2,\) in the system of integral equations \eqref{eq:SystemVolterraEqII} and prove that this system can be solved using the method of successive approximations.

First, consider the right-hand sides of these equations, i.e., the functions \(\Psi_i(s,t),\,i=1,2,\) which are represented by \(\Psi(s,t)\) and \(\Phi(s,t)\) in \eqref{eq:VolterraII} and \eqref{eq:PhiFormula}, respectively. We have already established that the function \(\Phi(s,t)\) is continuous in \(s\in[0,t)\) and satisfies the inequality \eqref{ineq:PhiEstimate}. Let us prove that the same inequality holds also for the function \(\Psi(s,t)\). Indeed, the validity of \eqref{ineq:PhiEstimate} for the first two terms in the expression for \(\Psi(s,t)\) is an immediate consequence of the property V of the functions \(q_i(s),\,i=1,2,\) and the inequality \eqref{ineq:PoissonEstimate} with (\(r=0,\,p=1\)). The third term is represented as the sum of two integrals (we denote them by \(I_1(s,t)\) and \(I_2(s,t)\)) with respect to the measure \(\mu\) over \(D_{1s}\) and \(D_{2s},\) respectively. To estimate the integral \(I_i(s,t),\,i=1,2,\) split the range of integration \(D_{is}\) into \(D_{is}^1=\{y:\, y\in D_{is},\,|y-h(s)|<1\}\) and \(D_{is}\setminus D_{is}^1\). Put
\[I_i(s,t)=\int\limits_{D_{is}^1}\Delta_y^{h(s)}u_{i0}(s,y,t)\mu(s,dy)+\int\limits_{D_{is}\setminus D_{is}^1}\Delta_y^{h(s)}u_{i0}(s,y,t)\mu(s,dy).\]
To estimate the first integral, we apply the mean value theorem, the inequality \eqref{ineq:PoissonEstimate} (with \(r=0,\,p=1\)), and the inequality \eqref{ineq:InfiniteMeasure}. To estimate the second one, we only use the inequality \eqref{ineq:PoissonEstimate} (with \(r=0,\,p=0\)) and the inequality \eqref{ineq:InfiniteMeasure}. We then get
\[
\begin{split}
|I_i(s,t)|&\leq C\|\varphi\|\bigg[(t-s)^{-\frac{1}{2}}\int\limits_{D_{is}^1}|y-h(s)|\mu(s,dy)+\int\limits_{D_{is}\setminus D_{is}^1}\mu(s,dy)\bigg]\leq\\
&\leq C\|\varphi\|(t-s)^{-\frac{1}{2}},\quad 0\leq s<t\leq T,\,i=1,2.\\\\
\end{split}
\]
This completes the proof of the inequality \eqref{ineq:PhiEstimate} for \(\Psi(s,t)\).

Having established the estimates for \(\Phi(s,t)\) and \(\Psi(s,t)\), in view of the properties of the coefficients of the equation \eqref{eq:ParabolicEq} and the parameters \(q_i,\,i=1,2\) in \eqref{eq:FellerWentzell}, we find that the functions \(\Psi_i(s,t),\,i=1,2,\) are continuous in \(s\in[0,t)\) and that the inequality
\begin{equation}\label{ineq:PsiiEstimate}
|\Psi_i(s,t)|\leq C_0\|\varphi\|(t-s)^{-\frac{1}{2}}
\end{equation}
holds for all \(0\leq s<t\leq T\), where \(C_0\) is a positive constant.

We now get down to studying the kernels \(N_{ij}(s,\tau),\,i=1,2,\) of the integral equations in \eqref{eq:SystemVolterraEqII}, which are represented by \(K_j(s,\tau)\) and \(R_j(s,\tau)\) in \eqref{eq:VolterraII} and \eqref{eq:E(s,t)I}, respectively. We have already established that the functions \(R_j(s,\tau)\) have integrable singularity when \(\tau=t\) and satisfy the estimate \eqref{ineq:RjEstimate}. From \eqref{ineq:DxGiEstimate}, it follows that the first term in the expression for \(K_j(s,\tau)\) allows the same singularity. Thus, it remains to investigate only the integrals with respect to the measure \(\mu\) in the formula for \(K_j(s,\tau),\) which we denote by \(\widetilde{I}_j(s,\tau),\,j=1,2.\) Taking into account \eqref{eq:FS} and splitting \(D_{js}\) into \(D_{js}^{\delta}\) and \(D_{js}\setminus D_{js}^{\delta},\, \delta>0\), we put
\begin{equation}\label{eq:WidetildeIj}
\begin{split}
\widetilde{I}_j(s,\tau)&=\int\limits_{D_{js}\setminus D_{js}^{\delta}}\Delta_y^{h(s)}G_j(s,y,\tau,h(\tau))\mu(s,dy)+\\
&+\int\limits_{D_{js}^{\delta}}\Delta_y^{h(s)}Z_{j1}(s,y,\tau,h(\tau))\mu(s,dy)+\\
&+\int\limits_{D_{js}^{\delta}}\Delta_y^{h(s)}Z_{j0}(s,y,\tau,h(\tau))\mu(s,dy)=\sum\limits_{k=1}^{3}\widetilde{I}_{j}^{(k)}(s,\tau).
\end{split}
\end{equation}
To estimate the first integral \(\widetilde{I}_{j}^{(1)}\) in \eqref{eq:WidetildeIj}, use \eqref{ineq:FSEstimate1} (with \(r=p=0\)) and the inequality \eqref{ineq:InfiniteMeasure}. We find that
\begin{equation}\label{ineq:WidetildeIj1}
|\widetilde{I}_j^{(1)}(s,\tau)|\leq C(\tau-s)^{-\frac{1}{2}}\int\limits_{D_{js}\setminus D_{js}^{\delta}}\mu(s,dy)\leq C(\delta)(\tau-s)^{-\frac{1}{2}},\quad j=1,2,
\end{equation}
where \(0\leq s<\tau\leq t\leq T,\,C(\delta)\) is a positive constant which depends on \(\delta\).

To estimate the second and the third integrals in \eqref{eq:WidetildeIj}, apply the mean value theorem, the inequality \eqref{ineq:InfiniteMeasure} and the estimates \eqref{ineq:FSEstimate2} and \eqref{ineq:FSEstimate1}. We have \((0\leq s<\tau\leq t\leq T)\)
\begin{align}
&|\widetilde{I}_j^{(2)}(s,\tau)|\leq C(\tau-s)^{-1+\frac{\alpha}{2}}\int\limits_{D_{js}^{\delta}}|y-h(s)|\mu(s,dy)\leq C(\delta)(\tau-s)^{-1+\frac{\alpha}{2}},\quad j=1,2,\label{ineq:WidetildeIj2}\\
&|\widetilde{I}_j^{(3)}(s,\tau)|\leq C(\tau-s)^{-1}\int\limits_{D_{js}^{\delta}}|y-h(s)|\mu(s,dy)\leq C(\delta)(\tau-s)^{-1},\quad j=1,2.\label{ineq:WidetildeIj3}
\end{align}

Combining the inequalities \eqref{ineq:WidetildeIj1}-\eqref{ineq:WidetildeIj3}, for \(K_j(s,\tau),\) we get the estimate of the form \eqref{ineq:WidetildeIj3}. This means that the kernels \(N_{ij}(s,\tau),\,i,j=1,2,\) of the system of Volterra integral equations of the second kind \eqref{eq:SystemVolterraEqII} are strongly singular since they contain the terms which have a non-integrable singularity at \(\tau=s\). Despite this, let us prove that the ordinary method of successive approximations can still be applied to this system of equations.

Thus, we look for solutions of the system of integral equations \eqref{eq:SystemVolterraEqII} of the form of the series
\begin{equation}\label{eq:Series}
V_i(s,t)=\sum\limits_{k=0}^{\infty}V_i^{(k)}(s,t),\quad i=1,2,
\end{equation}
where
\begin{equation}\nonumber
\begin{split}
&V_i^{(0)}(s,t)=\Psi_i(s,t),\quad i=1,2,\\
&V_i^{(k)}(s,t)=\sum\limits_{j=1}^2\int\limits_s^t N_{ij}(s,\tau)V_j^{(k-1)}(\tau,t)d\tau,\quad i=1,2,\, k=1,2,\ldots
\end{split}
\end{equation}

We get down to proving the convergence of the series \eqref{eq:Series}. For the function \(V_i^{(0)}(s,t)=\Psi_i(s,t),\, i=1,2,\) we have already established the estimate \eqref{ineq:PsiiEstimate}. To estimate other approximations in the series \eqref{eq:Series}, we return once again to the expressions for the kernels \(K_j(s,\tau),\,j=1,2,\) and use the following representation for the integral \(\widetilde{I}_j^{(3)}\) in \eqref{eq:WidetildeIj}:
\[
\widetilde{I}_j^{(3)}(s,\tau)=-[2\pi b_j(\tau,h(\tau))(\tau-s)]^{-\frac{1}{2}}\int\limits_{D_{js}^{\delta}}\mu(s,dy)\int\limits_0^1\frac{\partial}{\partial\theta}\exp\left\{-\frac{A(\theta,y,h(\tau),h(s))}{2b_j(\tau,h(\tau))(\tau-s)}\right\}d\theta,
\]
where
\[
A(\theta,y,h(\tau),h(s))=(1-\theta)(y-h(\tau))^2+\theta(h(s)-h(\tau))^2.
\]
After differentiating the integrand with respect to \(\theta\), we get
\begin{equation}\label{eq:WidetildeIj3}
\widetilde{I}_j^{(3)}(s,\tau)=\widetilde{I}_j^{(31)}(s,\tau)+\widetilde{I}_j^{(32)}(s,\tau),
\end{equation}
where
\begin{align}
\widetilde{I}_j^{(31)}(s,\tau)&=\frac{h(\tau)-h(s)}{\sqrt{2\pi}[b_j(\tau,h(\tau))(\tau-s)]^{\frac{3}{2}}}\times\nonumber\\
&\times\int\limits_{D_{js}^{\delta}}(y-h(s))\mu(s,dy)\int\limits_0^1\exp\left\{-\frac{A(\theta,y,h(\tau),h(s))}{2b_j(\tau,h(\tau))(\tau-s)}\right\}d\theta,\nonumber\\
\widetilde{I}_j^{(32)}(s,\tau)&=-\frac{1}{2\sqrt{2\pi}[b_j(\tau,h(\tau))(\tau-s)]^{\frac{3}{2}}}\int\limits_{D_{js}^{\delta}}(y-h(s))^2\mu(s,dy)\times\nonumber\\
&\times\int\limits_0^1\exp\left\{-\frac{A(\theta,y,h(\tau),h(s))}{2b_j(\tau,h(\tau))(\tau-s)}\right\}d\theta.\nonumber
\end{align}

Using the condition V and the inequality \eqref{ineq:InfiniteMeasure}, we get the estimate \eqref{ineq:WidetildeIj2} for the term \(\widetilde{I}_j^{(31)}(s,\tau)\) on the right-hand side of \eqref{eq:WidetildeIj3}. Taking into account \eqref{eq:WidetildeIj} and \eqref{eq:WidetildeIj3}, we can write \(N_{ij}(s,\tau)\) in the form
\begin{equation}\label{eq:Nij}
N_{ij}(s,\tau)=N_{ij}^{(1)}(s,\tau)+N_{ij}^{(2)}(s,\tau),\quad i,j=1,2,
\end{equation}
where
\[N_{ij}^{(2)}(s,\tau)=d_i(s)\widetilde{I}_j^{(32)}(s,\tau),\]
and the functions \(N_{ij}^{(1)}(s,\tau)\) are defined by the same formulas like the ones for \(N_{ij}(s,\tau)\) with the integrals \(\widetilde{I}_j(s,\tau)\) in the expression for \(K_j(s,\tau)\) replaced by the sum of the integrals \(\widetilde{I}_j^{(1)}(s,\tau),\,\widetilde{I}_j^{(2)}(s,\tau)\) and \(\widetilde{I}_j^{(31)}(s,\tau)\). As we can see, all the terms in the expression for \(N_{ij}^{(1)}(s,\tau)\) satisfy the inequalities \eqref{ineq:RjEstimate} or \eqref{ineq:WidetildeIj2}. Combining them, we obtain the following estimate for \(N_{ij}^{(1)}(s,\tau)\):
\begin{equation}\label{ineq:Nij1Estimate}
|N_{ij}^{(1)}(s,\tau)|\leq C_1(\delta)(\tau-s)^{-1+\frac{\alpha}{2}},\quad i,j=1,2,
\end{equation}
which holds with some constant \(C_1(\delta)\) in every domain of the form \(0\leq s<\tau\leq t\leq T\).

Next, to estimate \(N_{ij}^{(2)}(s,\tau)\) on the right-hand side of \eqref{eq:Nij}, we use the following inequalities for \(A(\theta,y,h(\tau),h(s))\) and \(d_i(s)\,(i=1,2)\):
\begin{equation}\label{ineq:Ad_iEstimates}
A(\theta,y,h(\tau),h(s))\geq\theta(1-\theta)(y-h(s))^2,\quad d_i(s)\leq\frac{B}{q_0}\left(\frac{B}{b}\right)^{\frac{1}{2}}.\\
\end{equation}
The constants \(B\) and \(b\) in \eqref{ineq:Ad_iEstimates} are the same as in I, and \[q_0=\min\limits_{s\in[0,T]}(q_1(s)+q_2(s)),\quad q_0>0.\] We get
\begin{equation}
\begin{split}\label{ineq:Nij2Estimate}
|N_{ij}^{(2)}(s,\tau)|&\leq \frac{B^{\frac{3}{2}}}{2q_0b^2\sqrt{2\pi}}(\tau-s)^{-\frac{3}{2}}\int\limits_{D_{js}^{\delta}}(y-h(s))^2\mu(s,dy)\times\\
&\times\int\limits_0^1\exp\left\{-\frac{\theta(1-\theta)(y-h(s))^2}{2B(\tau-s)}\right\}d\theta.
\end{split}
\end{equation}

We now estimate \(V_i^{(1)}(s,t),\,i=1,2,\) in \eqref{eq:Series}. Applying \eqref{eq:Nij}, \eqref{ineq:Nij1Estimate}, \eqref{ineq:Nij2Estimate}, the relations
\begin{align}
&\int\limits_s^t (t-\tau)^{-\frac{1}{2}}(\tau-s)^{-\frac{3}{2}}\exp\left\{-\frac{\theta(1-\theta)(y-h(s))^2}{2B(\tau-s)}\right\}d\tau\nonumber\\
&=\left(\frac{2\pi B}{\theta(1-\theta)(t-s)}\right)^{\frac{1}{2}}\frac{1}{|y-h(s)|}\exp\left\{-\frac{\theta(1-\theta)(y-h(s))^2}{2B(t-s)}\right\},\nonumber\\
&\int\limits_0^1\theta^{-\frac{1}{2}}(1-\theta)^{-\frac{1}{2}}d\theta=\pi\nonumber
\end{align}
and the inequality \eqref{ineq:InfiniteMeasure}, we find that \((0\leq s<t\leq T)\)
\begin{align}
&|V_i^{(1)}(s,t)|\leq C_0\|\varphi\|\bigg[2C_1(\delta)\int\limits_s^t (t-\tau)^{-\frac{1}{2}}(\tau-s)^{-1+\frac{\alpha}{2}}d\tau+\nonumber\\
&+\frac{B^{\frac{3}{2}}}{2q_0b^2\sqrt{2\pi}}\int\limits_{D_s^{\delta}}(y-h(s))^2\mu(s,dy)\int\limits_0^1 d\theta\int\limits_s^t(\tau-s)^{-\frac{3}{2}}\times\nonumber\\
&\times(t-\tau)^{-\frac{1}{2}}\exp\left\{-\frac{\theta(1-\theta)(y-h(s))^2}{2B(\tau-s)}\right\}d\tau\bigg]\leq\nonumber\\
&\leq C_0\|\varphi\|(t-s)^{-\frac{1}{2}}\left[\frac{2C_1(\delta)\Gamma(\frac{\alpha}{2})\Gamma(\frac{1}{2})}{\Gamma(\frac{1}{2}+\frac{\alpha}{2})}(t-s)^{\frac{\alpha}{2}}+m(\delta)\right],\quad i=1,2,\label{ineq:Vi1Estimate}
\end{align}
where \(\Gamma(\sigma)\) is the gamma function,
\[m(\delta)=\left(\frac{B}{b}\right)^2\frac{\pi}{2q_0}\max_{s\in[0,T]}\int\limits_{D_s^{\delta}}|y-h(s)|\mu(s,dy).\]
Furthermore, note that our assumptions on the measure \(\mu\) in \eqref{eq:FellerWentzell} allow to choose a positive \(\delta\) such that \(m(\delta)<1.\) Fix such \(\delta=\delta_0\) and put \(C_1=C_1(\delta_0),\,m_0=m(\delta_0)<1.\)

In view of the last remark, one can establish the following estimate for \(V_i^{(k)}(s,t),\,i=1,2\,(0\leq s<t\leq T)\) by induction on \(k\):
\begin{equation}\label{ineq:VikEstimate}
|V_i^{(k)}(s,t)|\leq C_0\|\varphi\|(t-s)^{-\frac{1}{2}}\sum\limits_{n=0}^k\binom{k}{n}h_{s,t}^{(k-n)}m_0^n,
\end{equation}
where
\[h_{s,t}^{(k)}=\frac{(2C_1\Gamma(\frac{\alpha}{2}))^k\Gamma(\frac{1}{2})}{\Gamma(\frac{1+k\alpha}{2})}(t-s)^{\frac{k\alpha}{2}}.\]
Hence, we have \((0\leq s<t\leq T)\)
\begin{align}
\sum\limits_{k=0}^{\infty}|V_i^{(k)}(s,t)|&\leq C_0\|\varphi\|(t-s)^{-\frac{1}{2}}\sum\limits_{k=0}^{\infty}\sum\limits_{n=0}^k\binom{k}{n}h_{s,t}^{(k-n)}m_0^n=\nonumber\\
&=C_0\|\varphi\|(t-s)^{-\frac{1}{2}}\sum\limits_{k=0}^{\infty}h_{s,t}^{(k)}\sum\limits_{n=0}^{\infty}\binom{k+n}{n}m_0^n=\nonumber\\
&=C_0\|\varphi\|(t-s)^{-\frac{1}{2}}\sum\limits_{k=0}^{\infty}\frac{(2C_1\Gamma(\frac{\alpha}{2}))^k\Gamma(\frac{1}{2})}{\Gamma(\frac{1}{2}+\frac{k\alpha}{2})(1-m_0)^{k+1}}(t-s)^{\frac{k\alpha}{2}},\quad i=1,2.\label{ineq:SumVik}
\end{align}

The inequality \eqref{ineq:SumVik} ensures the absolute and uniform convergence of the series \eqref{eq:Series} for \(s\in[0,t)\) and leads to the estimate
\begin{equation}\label{ineq:Vi}
|V_i(s,t)|\leq C\|\varphi\|(t-s)^{-\frac{1}{2}},\quad i=1,2,
\end{equation}
where \(0\leq s<t\leq T\) and \(C\) is a positive constant.

Thus, the formula \eqref{eq:Series} represents the unique solution of the system of integral equations \eqref{eq:SystemVolterraEqII}, which is continuous in the domain \(0\leq s<t\leq T\) and for which the inequality \eqref{ineq:Vi} holds.

From the estimates \eqref{ineq:FSEstimate1} (with \(r=p=0\)) and \eqref{ineq:Vi}, it follows that there exist the simple-layer potentials \(u_{i1}(s,x,t),\,i=1,2,\) in \eqref{eq:Solution}, and for them the inequality
\begin{equation}\label{ineq:ui1}
|u_{i1}(s,x,t)|\leq C\|\varphi\|,\quad i=1,2,\quad (s,x)\in \overline{S}_t,
\end{equation}
holds. It is obvious (see \eqref{ineq:PoissonEstimate}) that the same inequality is also true for the Poisson potentials \(u_{i0}(s,x,t),\,i=1,2,\) in \eqref{eq:Solution} and thus for the function \(u(s,x,t)\) itself. It remains to verify \eqref{eq:ui1Initial} for every fixed \(t\in(0,T]\) and any \(x\in\mathbb{R}\). To do this, we consider the function \(V_i(s,t),\,0\leq s<t,\,i=1,2,\) and study its behavior in a neighborhood of the point \(s=t\). Using the considerations similar to those leading to \eqref{ineq:Vi}, one can prove the following assertion: for every \(t\in(0,T]\) and \(\varepsilon>0\) there exists \(s_0\in[0,t)\) such that the inequality
\begin{equation}\label{ineq:ViE}
|V_i(s,t)|\leq\varepsilon C(t-s)^{-\frac{1}{2}}
\end{equation}
holds for all \(s\in[s_0,t)\), where the constant \(C\) is independent of \(\varepsilon\).

Let \(\varepsilon>0\) and choose \(s_0\) such that \eqref{ineq:ViE} holds for all \(s\in[s_0,t)\). Then, from \eqref{ineq:ViE} and \eqref{ineq:FSEstimate1}, we deduce that
\[
|u_{i1}(s,x,t)|\leq \varepsilon C\int\limits_s^t(\tau-s)^{-\frac{1}{2}}(t-\tau)^{-\frac{1}{2}}d\tau=\varepsilon C\pi,\quad i=1,2
\]
(the constant \(C\) is independent of \(\varepsilon\)) for all \(s\in[s_0,t)\) and \(x\in\mathbb{R}\). Because of the arbitrariness of \(\varepsilon\) this implies that \eqref{eq:ui1Initial} holds for \(x\in\mathbb{R}\).

Taking into account \eqref{eq:PoissonInitialCond} and, at the same time, the fact that the functions \(u_{i0}(s,x,t)\) and \(u_{i1}(s,x,t)\) satisfy the equation \eqref{eq:ParabolicEq} in the domain \((s,x)\in S_t^{(i)},\,i=1,2,\) we conclude that \(u(s,x,t)\) is the desired classical solution of the problem \eqref{eq:ParabolicEq}-\eqref{eq:FellerWentzell}.

The proof of Theorem \ref{th:Existence} is now complete.
\end{proof}

Now, let us prove the uniqueness theorem.
\begin{thm}[Uniqueness]\label{th:Uniqueness} Let the conditions I-IV hold. Then there is at most one solution of the conjugation problem \eqref{eq:ParabolicEq}-\eqref{eq:FellerWentzell} which is continuous and bounded in \(\overline{S}_t\).
\end{thm}
\begin{proof} Suppose that the problem \eqref{eq:ParabolicEq}-\eqref{eq:FellerWentzell} has two solutions \[u^{(1)}(s,x,t)=u_i^{(1)}(s,x,t),\quad u^{(2)}(s,x,t)=u_i^{(2)}(s,x,t),\quad (s,x)\in \overline{S}_t^{(i)},\,i=1,2,\] which are continuous and bounded in  \(\overline{S}_t\). Then the function \[\upsilon(s,x,t)=u^{(1)}(s,x,t)-u^{(2)}(s,x,t)=\upsilon_i(s,x,t),\quad (s,x)\in \overline{S}_t^{(i)},\,i=1,2,\]
where \(\upsilon_i(s,x,t)=u_i^{(1)}(s,x,t)-u_i^{(2)}(s,x,t),\) is the solution of the homogeneous conjugation problem \eqref{eq:ParabolicEq}-\eqref{eq:FellerWentzell} (with \(\varphi\equiv0\)), and is continuous and bounded in \(\overline{S}_t\). Note that each of the functions \(\upsilon_i,\,i=1,2,\) can be considered, at the same time, as the solution of the following first boundary value problem:
\begin{align}
&\frac{\partial\upsilon_i}{\partial s}+L_s^{(i)}\upsilon_i=0,\quad (s,x)\in S_t^{(i)},\,i=1,2,\label{eq:FirstBoundaryValue1}\\
&\lim\limits_{s\uparrow t}\upsilon_i(s,x,t)=0,\quad x\in \overline{D}_{it},\,i=1,2,\label{eq:FirstBoundaryValue2}\\
&\upsilon_i(s,h(s),t)=g_i(s,t),\quad 0\leq s<t\leq T,\,i=1,2,\label{eq:FirstBoundaryValue3}
\end{align}
where \(g_i(s,t)=\upsilon_{3-i}(s,h(s),t)+B_2\upsilon(s,h(s),t),\,i=1,2\). In view of the fact that \(B_2\upsilon=0\), we deduce that the function \(g_i(s,t)\) is continuous and bounded in \(\overline{S}_t^{(i)},\,i=1,2.\) However, under such conditions (see, e.g., \cite{kamynin4}), the first boundary value problem \eqref{eq:FirstBoundaryValue1}-\eqref{eq:FirstBoundaryValue3} has a unique classical solution which is continuous and bounded in \(\overline{S}_t^{(i)},\,i=1,2,\) and furthermore, can be represented by the formula \eqref{eq:Solution} with \(u_{i0}\equiv0,\,i=1,2\) (see, e.g., \cite{baderko2}). Thus, \(\upsilon_i(s,x,t),\,i=1,2,\) can be uniquely represented in the form \eqref{eq:Solution} with \(u_{i0}\equiv0,\,i=1,2,\) and \(V_i(s,t)\) are continuous in \(s\in[0,t),\) and represented by \(g_i(s,t)\). Further, in view of the considerations given in the proof of Theorem \ref{th:Existence}, it is easy to note that \(V_i(s,t),\,i=1,2,\) is also the solution of the homogeneous system of integral equations (\ref{eq:SystemVolterraEqII}) with \(\Psi_i(s,t)\equiv 0,\,i=1,2\). Because of the uniqueness of the solution of (\ref{eq:SystemVolterraEqII}) in the class of the continuous functions under consideration, we have \(V_i(s,t)\equiv0\,(i=1,2)\) and, hence, we get that \(\upsilon_i(s,x,t)\equiv0,\,i=1,2\). Therefore, \(u^{(1)}(s,x,t)=u^{(2)}(s,x,t)\) and the proof is complete.
\end{proof}

\section{Construction of the Markov process}

Theorem \ref{th:Existence} allows us to define a two-parameter family of operators \((T_{st})_{0\leq s\leq t\leq T}\) in \(C_b(\mathbb{R})\) by using the solution \(u(s,x,t)\) of the problem \eqref{eq:ParabolicEq}-\eqref{eq:FellerWentzell}. For \(0\leq s<t\leq T,\,x\in\mathbb{R}\) and \(\varphi\in C_b(\mathbb{R}),\) we put
\begin{equation}\label{eq:Semigroup}
T_{st}\varphi(x)=T_{st}^{(i0)}\varphi(x)+T_{st}^{(i1)}\varphi(x),\quad (s,x)\in [0,t)\times \overline{D}_{is},\,i=1,2,
\end{equation}
where \(T_{st}^{(i0)}\varphi(x)=u_{i0}(s,x,t),\,T_{st}^{(i1)}\varphi(x)=u_{i1}(s,x,t),\,i=1,2,\) the functions \(u_{i0}\) and \(u_{i1},\,i=1,2,\) are defined by the formulas \eqref{eq:Poisson} and \eqref{eq:Simple-layer}, respectively, and the densities \(V_i(s,t)\equiv V_i(s,t,\varphi),\,i=1,2,\) which are included in the simple-layer potentials \(u_{i1},\) represent the solution of the system of singular integral equations \eqref{eq:SystemVolterraEqII} to which the problem \eqref{eq:ParabolicEq}-\eqref{eq:FellerWentzell} is reduced. In addition, \(T_{tt}=I,\) where \(I\) is the identity operator, and for \(T_{st}\varphi(x),\) the estimate \eqref{ineq:ui1} holds for all \(0\leq s\leq t\leq T,\,x\in\mathbb{R}\).

The presence of the integral representation for the family of operators \((T_{st})_{0\leq s\leq t\leq T}\) allows us to to verify easily the following conditions:
\begin{enumerate}
\item[1)] if \(\varphi_n\in C_b(\mathbb{R}),\,n\in\mathbb{N},\,\sup\limits_{n}\|\varphi_n\|<\infty\) and for all \(x\in\mathbb{R},\,\lim\limits_{n\rightarrow\infty}\varphi_n(x)=\varphi(x)\), where \(\varphi\in C_b(\mathbb{R})\), then for all \((s,x)\in \overline{S}_t,\)
\[
\lim\limits_{n\rightarrow\infty}T_{st}\varphi_n(x)=T_{st}\varphi(x);
\]
\item[2)] for all \(0\leq s\leq\tau\leq t\leq T,\)
\[T_{st}=T_{s\tau}T_{\tau t};\]
\item[3)] \(T_{st}\varphi(x)\geq 0\) for all \((s,x)\in \overline{S}_t\) if \(\varphi\in C_b(\mathbb{R})\) and \(\varphi(x)\geq 0\);
\item[4)] The operators \(T_{st}\) are contractive, i.e., they do not increase the norm of element.
\end{enumerate}

Let us get down to proving these properties. The first one follows from the Lebesque bounded convergence theorem and the relation \(\lim\limits_{n\rightarrow\infty}V_i(s,t,\varphi_n)=V_i(s,t,\varphi),\,s\in[0,t),\,i=1,2,\) which holds for the solution of the system of integral equations \eqref{eq:SystemVolterraEqII}. The second property, which means that the family of operators \((T_{st})\) is a two-parameter semigroup, follows from the assertion of Theorem \ref{th:Uniqueness} on uniqueness of the solution of the problem \eqref{eq:ParabolicEq}-\eqref{eq:FellerWentzell}. Indeed, to find \(u(s,x,t)\) when \(u(t,x,t)=\varphi(x),\) one can do the following: solve the equation in the time interval \([\tau,t]\) and then solve it in the time interval \([s,\tau]\) starting with \(u(\tau,x,t)\) which was obtained; in other words, \(T_{st}\varphi=T_{s\tau}(T_{\tau t}\varphi),\,\varphi\in C_b(\mathbb{R})\), or \(T_{st}=T_{s\tau}T_{\tau t}\).

We now prove that the operators \(T_{st}\) remain a cone of nonnegative functions invariant.
\begin{lem}[Nonnegativeness]\label{lm:Nonnegativeness} If \(\varphi\in C_b(\mathbb{R})\) and \(\varphi(x)\geq 0\) for all \(x\in\mathbb{R}\), then  \(T_{st}\varphi(x)\geq 0\) for all \(0\leq s\leq t\leq T,\,x\in\mathbb{R}.\)
\end{lem}
\begin{proof}[Proof of Lemma \ref{lm:Nonnegativeness}] First we note that in view of the property 1), it suffices to consider the case when the nonnegative function \(\varphi\in C_b(\mathbb{R})\) has a compact support. From this and from Theorem \ref{th:Existence}, it follows that for such \(\varphi\), the function \(T_{st}\varphi(x)\) representing the solution of the problem \eqref{eq:ParabolicEq}-\eqref{eq:FellerWentzell} is bounded and continuous in the domain \((s,x)\in [0,t)\times\mathbb{R}\) and, furthermore, \(\lim\limits_{x\rightarrow\pm\infty} T_{st}\varphi(x)=0\) when \(0\leq s\leq t\leq T\). Next, if \(\varphi\equiv 0\), then the assertion of the lemma is obvious. We therefore suppose that the function \(\varphi\) is non-zero. Denote by \(\gamma\) the minimum of the function \(T_{st}\varphi(x)\) in the domain \((s,x)\in[0,t]\times\mathbb{R}\) and assume that \(\gamma<0\). According to the maximum principle for parabolic equations (see \cite[Ch.2]{friedman1}), we have that under our assumptions, the value \(\gamma\) is attained on the common boundary of the domains \(S_t^{(i)},\,i=1,2,\) i.e., when \((s,x)\in(0,t)\times\{h(s)\}\). Let \(s=s_0\) and \(x=h(s_0)\) for which \(T_{s_0 t}\varphi(x_0)=\gamma\). Then, for \(s=s_0,\) the following inequalities hold:
\begin{equation}
\begin{split}\label{ineq:Lemma}
&\frac{\partial T_{s_0t}\varphi}{\partial x}(x_0-0)\leq 0,\quad \frac{\partial T_{s_0t}\varphi}{\partial x}(x_0+0)\geq 0,\\
&\int\limits_{D_s}[T_{s_0t}\varphi(x_0)-T_{s_0t}\varphi(y)]\mu(s_0,dy)\leq 0.
\end{split}
\end{equation}
Moreover, it follows from Theorem 1 in \cite{kamynin3} (cf. Theorem 14 in \cite[Ch.II, \S4]{friedman1}) that in the first two estimates in \eqref{ineq:Lemma}, the equal signs should be excluded. If we recall, at the same time, the assumption IV on the parameters \(q_i,\,i=1,2,\) in the conjugation condition \eqref{eq:FellerWentzell}, it becomes clear that in the case of \(s=s_0\), \eqref{eq:FellerWentzell} cannot hold. The contradiction we arrived at indicates that \(\gamma\geq 0\). This completes the proof of the lemma.
\end{proof}

Finally, the property 4) of the operators \(T_{st}\) is an easy consequence of the property 3) and the following obvious relation:
\[T_{st}\varphi_0(x)\equiv1,\quad 0\leq s\leq t\leq T,\,x\in\mathbb{R}\quad\text{if}\quad\varphi_0\equiv1,\,x\in\mathbb{R}.\]

From the properties 1)-4) if follows (see, e.g., \cite[Ch.II, \S1]{dynkin}, \cite[Ch.I]{portenko2}) that the family of operators \((T_{st})_{0\leq s\leq t\leq T}\) is a semigroup associated with some inhomogeneous Feller process on \(\mathbb{R}\). If we denote by \(P(s,x,t,dy)\) its transition function, then \(T_{st}\varphi(x)\) can be represented as
\begin{equation}\label{eq:SemigroupAndTransitionFunc}
T_{st}\varphi(x)=\int\limits_{\mathbb{R}}\varphi(y)P(s,x,t,dy).
\end{equation}

In view of \eqref{eq:Series}, \eqref{eq:Semigroup} and \eqref{ineq:FSEstimate1}-\eqref{eq:FSadditional3}, we also find by direct calculation that for any \(s\in[0,T),\,\varphi\in C_b^2(\mathbb{R})\) and \(f\in C_0(\mathbb{R})\) (here \(C_0(\mathbb{R})\) denotes the set of all real continuous functions defined on \(\mathbb{R}\) with compact support) the transition function \(P(s,x,t,dy)\) satisfies the following relation:
\begin{equation}\label{eq:WeakGenerator}
\begin{split}
&\lim\limits_{t\downarrow s}\int\limits_{\mathbb{R}}f(x)\bigg[\frac{1}{t-s}(T_{st}\varphi(x)-\varphi(x))\bigg]dx=\int\limits_{\mathbb{R}}f(x)L_s\varphi(x)dx+\\
&+\frac{1}{2}(d_1(s)+d_2(s))\bigg[(q_2(s)-q_1(s))\varphi'(h(s))+\int\limits_{D_s}(\varphi(y)-\varphi(h(s)))\mu(s,dy)\bigg]f(h(s)).
\end{split}
\end{equation}
Here,
\begin{align}
&\label{eq:Ls}L_s\varphi(x)=\begin{cases}&L_s^{(i)}\varphi(x)\quad\text{for}\quad s\in[0,T],\,x\in D_{is},\,i=1,2,\\
&\sum\limits_{j=1}^2l_j(s)L_s^{(j)}\varphi(x)\quad\text{for}\quad s\in[0,T],\,x=h(s),\end{cases}\nonumber\\
&l_j(s)=\frac{q_j(s)\sqrt{b_{3-j}(s,h(s))}}{q_1(s)\sqrt{b_{2}(s,h(s))}+q_2(s)\sqrt{b_{1}(s,h(s))}},\quad j=1,2,\nonumber\\
&l_1(s)+l_2(s)=1.\nonumber
\end{align}

The right-hand side of \eqref{eq:WeakGenerator} represents the so-called weak (generalized) infinitesimal generator of the Markov process constructed above (cf. \cite{bogdan,xie}). Under our assumptions, the second term in \eqref{eq:WeakGenerator} can be reduced to
\begin{equation}
\begin{split}\label{eq:SecondTermWG}
&\frac{1}{2}(d_1(s)+d_2(s))\bigg[\varphi'(h(s))\bigg(q_2(s)-q_1(s)+\int\limits_{D_s^{\delta}}(y-h(s))\mu(s,dy)\bigg)+\\
&+\frac{1}{2}\varphi''(h(s))\int\limits_{D_s^{\delta}}(y-h(s))^2\mu(s,dy)+\int\limits_{D_s^{\delta}}\bigg(\varphi(y)-\varphi(h(s))-\varphi'(h(s))(y-h(s))-\\
&-\frac{1}{2}\varphi''(h(s))(y-h(s))^2\bigg)\mu(s,dy)+\int\limits_{D_s\setminus D_s^{\delta}}(\varphi(y)-\varphi(h(s)))\mu(s,dy)\bigg]f(h(s)).
\end{split}
\end{equation}

If we additionally assume the existence of the moments
\[
m_k(s)=\int\limits_{D_s}(y-h(s))^k\mu(s,dy)\in C[0,T],\quad k=1,2,
\]
for the measure \(\mu\), then \eqref{eq:SecondTermWG} can be rewritten as
\begin{equation}\label{eq:RewrittenSecondTermWG}
\begin{split}
&\frac{1}{2}(d_1(s)+d_2(s))\bigg[(q_2(s)-q_1(s)+m_1(s))\varphi'(h(s))+\frac{1}{2}m_2(s)\varphi''(h(s))+\\
&+\int\limits_{D_s}(\varphi(y)-\varphi(h(s))-\varphi'(h(s))(y-h(s))-\frac{1}{2}\varphi''(h(s))(y-h(s))^2)\mu(s,dy)]f(h(s)).
\end{split}
\end{equation}

Analyzing \eqref{eq:WeakGenerator} by taking into account \eqref{eq:SecondTermWG}, \eqref{eq:RewrittenSecondTermWG}, we see how the parameters in the Feller-Wentzell conjugation condition \eqref{eq:FellerWentzell} influence the different components of the bias of the diffusing particle (i.e., the drift, the diffusion and the random jump component) after it reaches the membrane. In the special case when \(\mu(s,D_s)\equiv0,\,s\in[0,T],\) the corresponding Markov process can be treated as the generalized diffusion process in the sense of the definition given in \cite[Ch.III, \S1]{portenko1} (see also \cite[Ch.I, \S2]{portenko2}). This means that when \(\mu\equiv0,\) the sample paths of this process are continuous and the diffusion coefficient equals
\begin{equation}\label{eq:DiffusionCoef}
b(s,x)=\begin{cases}b_i(s,x)\quad\text{for}\quad s\in[0,T],\,x\in D_{is},\,i=1,2,\\
\sum\limits_{j=1}^2l_j(s)b_j(s,h(s))\quad\text{for}\quad s\in[0,T],\,x=h(s),\end{cases}
\end{equation}
and the drift coefficient is the generalized function of the form
\[
a(s,x)+a_0(s)\delta(x-h(s)),
\]
where
\begin{align}
&\label{eq:DriftCoef1}a(s,x)=\begin{cases}a_i(s,x)\quad\text{for}\quad s\in[0,T],\,x\in D_{is},\,i=1,2,\\
\sum\limits_{j=1}^2l_j(s)a_j(s,h(s))\quad\text{for}\quad s\in[0,T],\,x=h(s),\end{cases}\\
&\label{eq:DriftCoef2}a_0(s)=\frac{1}{2}(d_1(s)+d_2(s))(q_2(s)-q_1(s)),
\end{align}
\(\delta(x-h(s))\) is the Dirac delta function concentrated at the point \(x=h(s)\).

The rigorous justification of the last assertion is based on direct verification of the following properties of the transition function \(P(s,x,t,dy)\) (in the case \(\mu\equiv0\)) :
\begin{enumerate}
\item[a)] for all \(s\in[0,T),\,x\in\mathbb{R},\)
\[\lim\limits_{t\downarrow s}\frac{1}{t-s}\int\limits_{\mathbb{R}}|y-x|^4P(s,x,t,dy)=0;\]
\item[b)] for every \(s\in[0,T)\) and any function \(f\in C_0(\mathbb{R})\),
\begin{align}
&\lim\limits_{t\downarrow s}\int\limits_{\mathbb{R}}f(x)\bigg[\frac{1}{t-s}\int\limits_{\mathbb{R}}(y-x)P(s,x,t,dy)\bigg]dx=\int\limits_{\mathbb{R}}a(s,x)f(x)dx+a_0(s)f(h(s)),\nonumber\\
&\lim\limits_{t\downarrow s}\int\limits_{\mathbb{R}}f(x)\bigg[\frac{1}{t-s}\int\limits_{\mathbb{R}}(y-x)^2P(s,x,t,dy)\bigg]dx=\int\limits_{\mathbb{R}}b(s,x)f(x)dx,\nonumber
\end{align}
\end{enumerate}
where the functions \(b(s,x),\,a(s,x)\) and \(a_0(s,x)\) are defined by the formulas \eqref{eq:DiffusionCoef}, \eqref{eq:DriftCoef1} and \eqref{eq:DriftCoef2}, respectively.

Finally, note that using the right-hand side of \eqref{eq:WeakGenerator}, one can also find the ordinary weak infinitesimal generator of the constructed Markov process (see \cite[Ch.V]{dynkin}, \cite[Ch.II]{friedman2}). This generator (we denote it by \(A_s\)) is defined on functions \(\varphi\in C_b(\mathbb{R}^2)\) such that
\begin{align}
&L_s^{(1)}\varphi(h(s))=L_s^{(2)}\varphi(h(s)),\nonumber\\
&(q_2(s)-q_1(s))\varphi'(s)+\int\limits_{D_s}(\varphi(y)-\varphi(h(s)))\mu(s,dy)=0
\end{align}
and, for them, it holds,
\[
A_s\varphi(x)=\lim\limits_{t\downarrow s}\frac{1}{t-s}(T_{st}\varphi(x)-\varphi(x))=L_s\varphi(x)
\]
for all \(x\in\mathbb{R},\, 0\leq s<T\).

Such a result is obviously expected. It is connected directly with the very formulation of the parabolic conjugation problem \eqref{eq:ParabolicEq}-\eqref{eq:FellerWentzell} which we consider (see the introductory part of this paper as well as the corresponding comments in Section 1).

Thus, the following theorem is a conclusion of the second part of our research:
\begin{thm}[Construction of the process]\label{th:ProcessConstruction}
Let the conditions of Theorem \ref{th:Existence} be satisfied. Then the two-parameter family of the operators \((T_{st})_{0\leq s\leq t\leq T}\) defined by \eqref{eq:Semigroup} is the semigroup associated with the inhomogeneous Markov process on \(\mathbb{R},\) the transition function  \(P(s,x,t,dy)\) of which satisfies the relation \eqref{eq:WeakGenerator}.
\end{thm}

%%%%%%%%%%%%%%%%%%%%%%%%%%%%%%%%%%%%%%%%%%%%%%%%%%%%%%%%%%%%%%%%%%%
%%                                                               %%
%% Use the two commands below for producing your bibliography    %%
%% with bibtex, then comment again the commands and include the  %%
%% content of the .bbl file in this file below the commands.     %%
%%                                                               %%
%%%%%%%%%%%%%%%%%%%%%%%%%%%%%%%%%%%%%%%%%%%%%%%%%%%%%%%%%%%%%%%%%%%

%\bibliographystyle{amsplain}
%\bibliography{yourbibfilename}

% add below the content of your .bbl file produced by bibtex.

%%%%%%%%%%%%%%%%%%%%%%%%%%%%%%%%%%%%%%%%%%%%%%%%%%%%%%%%%%%%%%%%%%%
%%                                                               %%
%% You may add acknowledgments (optional).                       %%
%%                                                               %%
%%%%%%%%%%%%%%%%%%%%%%%%%%%%%%%%%%%%%%%%%%%%%%%%%%%%%%%%%%%%%%%%%%%
%
%\ACKNO{We are grateful to Martin Hairer who provided a nice \texttt{MR} macro and to S\'ebastien Gou\"ezel for his useful comments on the internals of the class %file.}

%%%%%%%%%%%%%%%%%%%%%%%%%%%%%%%%%%%%%%%%%%%%%%%%%%%%%%%%%%%%%%%%%%%
%%                                                               %%
%% You have reached the end of your document.                    %%
%%                                                               %%
%%%%%%%%%%%%%%%%%%%%%%%%%%%%%%%%%%%%%%%%%%%%%%%%%%%%%%%%%%%%%%%%%%%

\end{document}